\documentclass[12pt,reqno]{amsart}
\usepackage{fullpage}
\usepackage{amsmath}
\usepackage{mathrsfs,amssymb,graphicx,verbatim,hyperref}
\usepackage{paralist}
\renewcommand{\setminus}{\smallsetminus}
\usepackage{ifthen}
\newcommand{\U}{\mathscr{U}}

\newcommand{\ifsodaelse}[2]{\ifthenelse{\isundefined{\SODAF}}{#2}{#1}}
\renewcommand{\theta}{\vartheta}
\ifsodaelse    {\usepackage{ltexpprt}\usepackage{amsmath}}{}
\usepackage{mathrsfs,amssymb,graphicx,verbatim}
\usepackage{paralist}

\newcommand\remove[1]{}

\newcommand{\rnote}[1]{}
\newcommand{\jnote}[1]{}

\newcommand{\e}{\varepsilon}
\newcommand{\R}{\mathbb{R}}

\newcommand{\N}{\mathbb{N}}

\newcommand{\Lip}{\mathrm{Lip}}
\newcommand{\vol}{\mathrm{vol}}

\newtheorem{theorem}{Theorem}[section]
\newtheorem{lemma}[theorem]{Lemma}

\newtheorem{remark}{Remark}[section]

\newtheorem{conjecture}{Conjecture}
\newtheorem{question}[conjecture]{Question}
\newcommand{\eqdef}{\stackrel{\mathrm{def}}{=}}
\date{}
\renewcommand{\t}{\theta}
\renewcommand{\le}{\leqslant}
\renewcommand{\ge}{\geqslant}

\include{psfig}

\renewcommand{\epsilon}{\varepsilon}

\theoremstyle{remark}

\renewcommand{\phi}{\varphi}

\title{Discretization and affine approximation\\ in high dimensions}
\thanks{S.~L. was supported by NSF grant CCF-0832795. A.~N.
was supported by NSF grant CCF-0832795, BSF grant
2006009, and the Packard Foundation. Some of this work was completed when both authors were in residence at the MSRI Quantitative Geometry program.}

\author{Sean Li}
\address{Courant Institute, New York University, New York NY 10012}
\email{seanli@cims.nyu.edu}

\author{Assaf Naor}
\address{Courant Institute, New York University, New York NY 10012}
\email{naor@cims.nyu.edu}

\date{}

\begin{document}
\maketitle
\begin{abstract}
Lower estimates are obtained for the macroscopic scale of affine approximability of vector-valued Lipschitz functions on finite dimensional normed spaces, completing the work of Bates, Johnson, Lindenstrass, Preiss and Schechtman. This yields a new approach to Bourgain's discretization theorem  for superreflexive targets.
\end{abstract}


\section{Introduction}

Let $X, Y$ be Banach space with (closed) unit balls $B_X,B_Y$, respectively. For $\e\in (0,1)$ define $r^{X\to Y}(\e)$ to be the supremum over those $r\in (0,1]$ such that  for every Lipschitz function $f:B_X\to Y$  there exists $y\in X$ and $\rho\in [r,\infty)$ such that $y+\rho B_X\subseteq B_X$, and there exists an affine mapping $A:X\to Y$ satisfying
\begin{equation}\label{eq:def r}
\sup_{z\in y+\rho B_X}\frac{\left\|f(z)-A(z)\right\|}{\rho}\le \e\|f\|_{\Lip},
\end{equation}
where $\|f\|_{\Lip}$ is the Lipschitz constant of $f$.
If no such $r\in (0,1]$ exists then set $r^{X\to Y}(\e)=0$. We call $r^{X\to Y}(\cdot)$ the {\bf modulus of affine approximability} corresponding to $X,Y$.

The assertion $r^{X\to Y}(\e)\ge r$ means that { every} Lipschitz function on the unit ball of $X$ is $\e$-close (after appropriate normalization) to an affine function on some sub-ball of radius $r$. Thus, while a differentiation statement corresponds to an assertion about the infinitesimal regularity of a function, bounding $r^{X\to Y}(\e)$ from below corresponds to proving a quantitative differentiation theorem about the regularity of Lipschitz functions on a maroscopic scale. This statement isn't quite precise, since there is no requirement of the affine mapping $A$ in~\eqref{eq:def r} to have any relation to the derivative of $f$ at $y$, but it turns out that for interesting applications it suffices (and necessary) to allow for an arbitrary affine approximation of $f$.

Bates, Johnson, Lindenstrass, Preiss and Schechtman introduced the above affine approximability problem in~\cite{BaJoLi}, where it was shown to have applications to the theory of nonlinear quotient mappings. Following~\cite{BaJoLi} we say that the space of Lipschitz mappings $$\Lip(X,Y)\eqdef\{f:X\to Y:\ \|f\|_\Lip<\infty\}$$ has the Uniform Approximation by Affine Property (UAAP) if $r^{X\to Y}(\e)>0$ for all $\e\in (0,1)$. A beautiful theorem of~\cite{BaJoLi} asserts that $\Lip(X,Y)$ has the UAAP if and only if one of the spaces $\{X,Y\}$ is finite dimensional, and the other space is superreflexive. Recall that a Banach space $Z$ is superreflexive if any Banach space that is finitely representable in $Z$ must be reflexive; equivalently all the ultrapowers of $Z$ are reflexive\footnote{See~\cite{DJT95} for background on finite representability and ultrapowers of Banach spaces.}. Due to deep works of James~\cite{Jam64,Jam72}, Enflo~\cite{Enf72} and Pisier~\cite{Pis75}, we know that $Z$ is superreflexive if and only if it admits an equivalent norm $||\cdot||$ for which there exist $p\in [2,\infty)$ and $K\in [1,\infty)$ such that
\begin{equation}\label{eq:4 point}
\forall\ x,y\in Z,\quad 2\|x\|^p+\frac{2}{K^p}\|y\|^p\le \|x+y\|^p+\|x-y\|^p.
\end{equation}
A norm that satisfies~\eqref{eq:4 point} is said to be uniformly convex of power type $p$; readers who are not familiar with the theory of superreflexivity can take the above renorming statement as the definition of superreflexivity. For concreteness, we recall~\cite{Fig76,BCL94} that for $q\in (1,\infty)$ the usual norm on an $L_q(\mu)$ space satisfies~\eqref{eq:4 point}  with $p=\max\{q,2\}$ and $K=\max\{1/\sqrt{q-1},1\}$.

Assume from now on that $n=\dim X<\infty$ and $Y$ is superreflexive. The theorem of Bates, Johnson, Lindenstrass, Preiss and Schechtman  says that  $r^{X\to Y}(\e),r^{Y\to X}(\e)>0$ for every $\e\in (0,1)$. The proof in~\cite{BaJoLi} of  $r^{Y\to X}(\e)>0$  is effective, yielding a concrete lower bound on $r^{Y\to X}(\e)$. This lower bound is quite small: a $O(n)$-fold iterated exponential of $-1/\e$ and  geometric parameters that measure the degree to which $Y$ is superreflexive. We leave the investigation of the true asymptotic behavior of $r^{Y\to X}(\e)$ as an interesting open problem.

Our main purpose here is to obtain a concrete lower bound on $r^{X\to Y}(\e)$. While we are partly motivated by an application of such bounds to Bourgain's discretization problem, as will be described in Section~\ref{sec:bourgain}, our main motivation is that the proof in~\cite{BaJoLi} of the estimate $r^{X\to Y}(\e)>0$  proceeds by contradiction using an ultrapower argument, and as such it does not yield any concrete quantitative information on $r^{X\to Y}(\e)$.

We briefly recall the argument of~\cite{BaJoLi}. The contrapositive assumption $r^{X\to Y}(\e)=0$ means that for every $k\in \N$ there is a $1$-Lipschitz function $f_k:B_X\to Y$ with $f_k(0)=0$ such that for all balls $y+rB_X\subseteq B_X$ with $r\ge 1/k$ and for all affine mappings $A:X\to Y$ we have $\|f_k-A\|_{L_\infty(y+rB_X)}\ge \e r$. Let $\mathscr U$ be a free ultrafilter on $\N$ and consider the mapping $f_\U:B_X\to Y_\U$ given by $f(x)=(f_k(x))_{k=1}^\infty$. Here $Y_\U$ denotes the ultrapower of $Y$; since $Y$ is superreflexive we are ensured that $Y_\U$ is reflexive. A moment of thought reveals that $f_\U$ is $1$-Lipschitz yet it cannot have a point of differentiability. This contradicts the fact~\cite{Gel38} that reflexive spaces have the Radon-Nikod\'ym property  (see~\cite[Ch.~5]{BL00}), and hence $f_\U$ is differentiable almost everywhere. Due to the ineffectiveness of this argument, the estimation of the fundamental parameter $r^{X\to Y}(\e)$ is a basic question that~\cite{BaJoLi} left open. This problem is resolved here via the following theorem.

\begin{theorem}\label{thm:main intro}
Fix $n\in \N$, $p\in [2,\infty)$ and $K\in [1,\infty)$. Assume that $n=\dim X<\infty$ and the norm of $Y$ satisfies~\eqref{eq:4 point}. Then for all $\e\in \left(0,\frac12\right)$ we have
\begin{equation}\label{eq:our estimate r}
r^{X\to Y}(\e)\ge\left\{\begin{array}{ll} \e^{(16K/\e)^p} &
\mathrm{if}\ n=1,\\\e^{K^pn^{20(n+p)}/\e^{2p+2n-2}}&\mathrm{if}\
n\ge 2.\end{array}\right.
\end{equation}
\end{theorem}

In Section~\ref{sec:example} we present an example showing that if $\e\in (0,\frac12)$ and $p\in [2,\infty)$ then for $X_0=\ell_2^n$  and $Y_0=\ell_2(\ell_p)$ we have $$r^{X_0\to Y_0}(\e)\le \frac{1}{\sqrt{n}}e^{-(\kappa/\e)^p},$$ where $\kappa\in (0,\infty)$ is a universal constant. Note that $\ell_2(\ell_p)$ satisfies~\eqref{eq:4 point}; see~\cite{Fig76}. Thus, when $n=1$ Theorem~\ref{thm:main intro} is quite sharp as $\e\to 0$  (up to a $\log(1/\e)$ term in the exponent), but it remains a very interesting open problem to determine the asymptotic behavior of $r^{X\to Y}(\e)$ as $n\to \infty$. It is worthwhile to single out the purely Hilbertian special case of this problem.
\begin{question}\label{q:hilbertian}
What is the asymptotic behavior of $r^{\ell_2^n\to \ell_2}\left(\frac12\right)$ as $n\to \infty$?
\end{question}

We note that despite the fact that the gap between~\eqref{eq:our estimate r} and the above upper bound on $r^{X\to Y}(\e)$ is very large as $n\to \infty$, the lower estimate on $r^{X\to Y}(\e)$ in~\eqref{eq:our estimate r} is sufficiently strong to match the best-known bound in Bourgain's discretization theorem for superreflexive targets; see Section~\ref{sec:bourgain}.

In our forthcoming  article~\cite{HLN12}, written jointly with Tuomas Hyt\"onen, we study a natural variant of the UAAP by replacing the $L_\infty$ requirement in~\eqref{eq:def r} by
$$
\left(\frac{1}{\rho^n\vol(B_X)}\int_{y+\rho B_X}\left(\frac{\|f(z)-A(z)\|}{\rho}\right)^p dz\right)^{1/p}\le \e \|f\|_{\Lip}.
$$
In this setting, we obtain in~\cite{HLN12} asymptotically stronger  lower bounds on $\rho$ when $Y$ is a UMD Banach space (see~\cite{Bur01} for a detailed discussion of UMD spaces). Unlike our proof of Theorem~\ref{thm:main intro}, which is entirely geometric, the arguments in~\cite{HLN12} are based on  vector-valued Littlewood-Paley theory.

\subsection{Bourgain's discretization theorem}\label{sec:bourgain}
Let $(X,d_X)$ and $(Y,d_Y)$ be metric spaces. The distortion of $X$ in $Y$, denoted $c_Y(X)$, is the infimum over those $D\in [1,\infty]$ for which there exists $f:X\to Y$ and $s\in (0,\infty)$ satisfying $$
\forall\ x,y\in X,\quad sd_X(x,y)\le d_Y(f(x),f(y))\le Dsd_X(x,y).$$

Suppose now that  $(X,\|\cdot\|_X)$ and $(Y,\|\cdot\|_Y)$ are normed spaces with $\dim(X)<\infty$ and $\dim(Y)=\infty$. For $\e\in [0,1)$ let $\delta_{X\hookrightarrow Y}(\e)$ be the supremum over those $\delta\in (0,1)$ such that every $\delta$-net $\mathcal{N}_\delta$ of $B_X$ satisfied $c_Y(\mathcal{N}_\delta)\ge (1-\e)c_Y(X)$.

A classical theorem of Ribe asserts that $\delta_{X\hookrightarrow Y}(\e)>0$ for all $\e\in (0,1)$. A different proof of this fact, due to Heinrich and Mankiewicz,  was obtained in~\cite{HM82}. Bourgain~\cite{Bou87} discovered yet another proof of the positivity of $\delta_{X\hookrightarrow Y}(\e)$, which, unlike previous proofs, yields the following concrete estimate, known as {\em Bourgain's discretization theorem}.
\begin{equation}\label{eq:Bourgain}
\delta_{X\hookrightarrow Y}(\e)\ge e^{-(n/\e)^{O(n)}}.
\end{equation}

It is an intriguing open question to determine the asymptotic behavior of the best possible lower bound on $\inf\{\delta_{X\hookrightarrow Y}(\e):\ \dim(X)=n\}$. This is of interest even for special classes of normed spaces $Y$, though there has been little progress on this problem besides the improved estimate $\delta_{X\hookrightarrow L_p}(\e)\gtrsim \e^2/n^{5/2}$, which was obtained in~\cite{GNS11} (here $p\in [1,\infty)$ and the implied constant is independent of $p$).

We shall now describe a different approach to Bourgain's discretization theorem based on Theorem~\ref{thm:main intro}. If an affine mapping is bi-Lipschitz on a fine enough net of a ball $y+\rho B_X$ then it is also bi-Lipschitz on all of $X$. It is therefore natural to approach the problem of estimating $\delta_{X\hookrightarrow Y}(\e)$ by first extending the embedding of the net $\mathcal{N}_\delta$ to a Lipschitz function defined on all of $X$, and then finding a large enough ball on which the extended function is approximately affine. By the theorem of Bates, Johnson, Lindensrauss, Preiss and Schechtman, for this strategy to work we need $Y$ to be superreflexive. Bourgain's discretization theorem is interesting even for superreflexive targets, and moreover the estimate~\eqref{eq:Bourgain} is the best known estimate even with this additional restriction on $Y$. It turns out that the above strategy, when combined with our estimate~\eqref{eq:our estimate r}, suffices to match Bourgain's bound~\eqref{eq:Bourgain} when $Y$ is superreflexive. The details of this link between Theorem~\ref{thm:main intro} and~\eqref{eq:Bourgain} are explained below.

Fix $\e\in (0,1)$, $p\in [2,\infty)$ and $K\in [1,\infty)$. Suppose that $\dim(X)=n\ge 2$ and the norm of $Y$ satisfies the uniform convexity condition~\eqref{eq:4 point}. Set
\begin{equation}\label{eq:choose delta}
\delta=e^{-K^p(n/\e)^{C(n+p)}},
\end{equation}
where $C\in (1,\infty)$ is a universal constant that will be determined later.

Let $\mathcal{N}_\delta$ be a $\delta$-net of $B_X$ and write $D=c_Y(\mathcal{N}_\delta)$. Note that, by John's theorem~\cite{Jo}, we have the a priori bound $D\le n$. Take $f:\mathcal{N}_\delta\to Y$ satisfying
\begin{equation}\label{eq:f assumption}
\forall\ x,y\in \mathcal{N}_\delta,\quad \|x-y\|_X\le \|f(x)-f(y)\|_Y\le \left(1+\frac{\e}{16}\right)D\|x-y\|_X.
\end{equation}
 By a Lipschitz extension theorem of Johnson, Lindenstrauss and Schechtman~\cite{JLS86}, there exists $F:X\to Y$ that coincides with $f$ when restricted to $\mathcal{N}_\delta$, and $\|F\|_{\Lip}\le cnD$, where $c\in (1,\infty)$ is a universal constant.

 By Theorem~\ref{thm:main intro} there exist $y\in X$, $z\in Y$, a linear mapping $T:X\to Y$, and a radius
\begin{equation}\label{eq:lower rho}
\rho\ge \e^{K^pn^{20(n+p)}(32cnD/\e)^{2p+2n-2}}.
\end{equation}
 such that $y+\rho B_X\subseteq B_X$ and
\begin{equation}\label{eq:F approximation}
\forall\ x\in y+\rho B_X,\quad \|F(x)-z-Tx\|_Y\le \frac{\e}{32}\rho.
\end{equation}
Note that it follows from~\eqref{eq:lower rho} that we can choose the constant $C$ in~\eqref{eq:choose delta} so that
\begin{equation}\label{eq:lower rho 2}
\rho\ge \frac{64n\delta}{\e}.
\end{equation}

Fix $u\in X$ with $\|u\|_X=1$. Choose $v,w\in \mathcal{N}_\delta\cap (y+\rho B_X)$ such that $\|v-y\|_X\le \delta$ and $\|w-y-\frac{\rho}{2}u\|_X\le \delta$. Thus $\|w-v-\frac{\rho}{2}u\|_X\le 2\delta$, and  consequently $\|w-v\|_X\in [\rho/2-2\delta,\rho/2+2\delta]$. Using the fact that $F$ extends $f$,
\begin{multline*}
\|Tw-Tv\|_Y\stackrel{\eqref{eq:F approximation}}{\le} \|f(w)-f(v)\|_Y+\frac{\e\rho}{16}\stackrel{\eqref{eq:f assumption}}{\le} \left(1+\frac{\e}{16}\right)D\|w-v\|_X+\frac{\e\rho}{16}\\ \le \left(1+\frac{\e}{16}\right)D\left(\frac{\rho}{2}+2\delta\right)+\frac{\e\rho}{16}\stackrel{\eqref{eq:lower rho 2}}{\le} \left(1+\frac{\e}{4}\right)\frac{\rho}{2}D.
\end{multline*}
Hence
$\|Tu\|_Y\le \frac{2}{\rho}\|Tw-Tv\|_Y+\frac{2\|T\|}{\rho}\left\|w-v-\frac{\rho}{2}u\right\|_X\le
\left(1+\frac{\e}{4}\right)D+\frac{4\delta\|T\|}{\rho}$. Since this holds for all unit vectors $u\in X$,
\begin{equation}\label{eq:norm T}
\|T\|\le \frac{1+\e/4}{1-4\delta/\rho}D\le \left(1+\frac{\e}{2}\right)D\le 2n.
\end{equation}
Now,
\begin{equation*}
\|Tw-Tv\|_Y\stackrel{\eqref{eq:F approximation}}{\ge} \|f(w)-f(v)\|_Y-\frac{\e\rho}{16}\stackrel{\eqref{eq:f assumption}}{\ge} \|w-v\|_X-\frac{\e\rho}{16} \ge \frac{\rho}{2}-2\delta-\frac{\e\rho}{16}\ge \left(1-\frac{\e}{4}\right)\frac{\rho}{2}.
\end{equation*}
Hence,
$$
\|Tu\|_Y\ge \frac{2}{\rho}\|Tw-Tv\|_Y-\frac{2\|T\|}{\rho}\left\|w-v-\frac{\rho}{2}u\right\|_X\stackrel{\eqref{eq:norm T}}{\ge}
1-\frac{\e}{4}-\frac{8n\delta}{\rho}\stackrel{\eqref{eq:lower rho 2}}{\ge} 1-\frac{\e}{2}.
$$
We have proved that
$
c_Y(X)\le \frac{1+\e/2}{1-\e/2}D=\frac{1+\e/2}{1-\e/2}c_Y(\mathcal{N}_\delta)\le \frac{1}{1-\e}c_Y(\mathcal{N}_\delta).
$
 Thus, recalling the choice of $\delta$ in~\eqref{eq:choose delta},
 \begin{equation}\label{eq:the hook delta}
 \delta_{X\hookrightarrow Y}(\e)\ge e^{-K^p(n/\e)^{C(n+p)}}.
 \end{equation}

\begin{remark} {\em In fact, we have the general estimate
  \begin{equation}\label{eq:delta r relation}
  \delta_{X\hookrightarrow Y} (\e)\ge \frac{\e}{n}\cdot r^{X\to Y}\left(\frac{\kappa\e}{c_Y(X)}\right),
   \end{equation}
   where $\kappa\in (0,\infty)$ is a universal constant. This estimate follows from a more careful application of the above reasoning. Specifically, we used the Lipschitz extension theorem of Johnson, Lindenstrauss and Schechtman~\cite{JLS86} to obtain the function $F$. This theorem ignores the fact that $f$ was defined on a $\delta$-net: it would apply equally well if $f$ were defined on {\em any} subset of $X$. One can exploit the additional information that the domain of $f$ is a net by invoking an approximate Lipschitz extension theorem of Bourgain~\cite{Bou87}. This theorem states that for every $\tau\in (20\delta,1)$ there exists a function $F_\tau:X\to Y$ such that $\|F_\tau(x)-f(x)\|_Y\le \tau$ for every $x\in \mathcal{N}_\delta$ and the Lipschitz constant of $F_\tau$ on $\frac12 B_X$ is $(1+O(n\delta/\tau))(1+\e/16)D$ (this formulation of Bourgain's approximate extension theorem is not stated explicitly in~\cite{Bou87}, but it easily follows from the argument in~\cite[Sec.~3]{GNS11}). Now, one can deduce~\eqref{eq:delta r relation} by repeating mutatis mutandis the above proof while optimizing over $\tau$. We omit the details since the resulting estimate, when applied to our bounds~\eqref{eq:our estimate r}, only affects the constant $C$ in~\eqref{eq:the hook delta}.}
\end{remark}

\section{Proof of Theorem~\ref{thm:main intro} when $n=1$}\label{sec:n=1}

The proof of Theorem~\ref{thm:main intro} when $n=1$ follows well-established metric differentiation methodology. This type of reasoning, also known as the approximate midpoint argument, seems to have been first used by Enflo in his classical proof that $L_1$ and $\ell_1$ are not uniformly homeomorphic; see~\cite{Ben85}. The basic idea is that a Lipschitz function $f:\R\to Y$ must map the midpoints between many pairs of points $x,y\in \R$ to ``almost midpoints" of $f(x)$ and $f(y)$. See Chapter 10 of~\cite{BL00} for a precise formulation of this principle. One can iterate this idea to deduce that $f$ must map many ``discretized geodesic segments" to ``discretized almost geodesics". Such an iteration of the midpoint argument is contained in e.g.~\cite[Prop.~1.4.9]{Nao98}, and a striking recent application of this type of reasoning can be found in~\cite{EFW07}. When the target space is uniformly convex, approximate geodesics must be close to straight lines. This rigidity statement explains why one can hope to find a macroscopically large region on which $f$ is almost affine. In order to obtain good quantitative control on the size of such a region one  follows the general strategy that is explained in Appendix~2 of~\cite{CKN09}. Using the terminology of~\cite{CKN09}, the ``coercive quantity" in our setting is the functional $E_m^{a,b}(\cdot)$ defined below.

In Section~\ref{sec:n>1} we build on the tools developed in this section to deduce Theorem~\ref{thm:main intro} when $n\ge 2$. In this setting one must deal with higher-dimensional phenomena, and to obtain good dimension-dependent bounds. We will argue that there must exist a cube on which a given Lipschitz function maps all axis-parallel discretized line segments to almost-straight lines. This does not imply that the function itself is almost affine on the same cube: at best it means that it is almost multi-linear. Therefore an additional argument is needed in order to find a scale on which the function is almost affine. Moreover, to obtain good control on this scale we reason about a more complicated (two-parameter) coercive quantity; see~\eqref{eq:two parameter H}.

\begin{lemma}\label{lem:use UC}
Fix $p\in [2,\infty)$. Suppose that $(Y,\|\cdot\|_Y)$ is a Banach space satisfying the uniform convexity condition~\eqref{eq:4 point}. Fix $a,b\in \R$ with $a<b$ and $h:[a,b]\to Y$. For $m\in \N\cup\{0\}$ define
\begin{equation}\label{eq:def Eab}
E_m^{a,b}(h)\eqdef \frac{1}{2^m}\sum_{k=0}^{2^m-1}
\left\|\frac{h(a+(k+1)2^{-m}(b-a))-h(a+k2^{-m}(b-a))}{2^{-m}(b-a)}\right\|_Y^p.
\end{equation}
Then
\begin{equation*}\label{eq:lower E}
E_m^{a,b}(h)\ge \frac{\left\|h(b)-h(a)\right\|_Y^p}{(b-a)^p}+\frac{1}{(2K)^p}\max_{k\in \{0,\ldots,2^m\}}
\frac{\left\|h\left(a+\frac{k}{2^m}(b-a)\right)-L_h^{a,b}\left(a+\frac{k}{2^m}(b-a)\right)\right\|_Y^p}
{(b-a)^p},
\end{equation*}
where $K\in (0,\infty)$ is the constant appearing in~\eqref{eq:4 point} and $L_h^{a,b}(h):[a,b]\to Y$ is the linear interpolation of the values of $h$ on the endpoints of the interval $[a,b]$, i.e.,
\begin{equation}\label{eq:def linear interpolation}
\forall t\in \R,\quad
L_h^{a,b}(t)\eqdef\frac{t-a}{b-a}h(b)+\frac{b-t}{b-a}h(a).
\end{equation}
\end{lemma}

\begin{proof} We may assume without loss of generality that $a=0$ and $b=1$. In this case, denote for the sake of simplicity $E_m^{0,1}(h)=E_m(h)$ and $L_h^{0,1}=L_h$. We will actually prove the following slightly stronger statement by induction on $m$: for every $k\in \{0,\ldots,2^m\}$,
\begin{equation}\label{eq:better inductive}
E_m(h)\ge \|h(1)-h(0)\|_Y^p+\frac{2^p}{K^p\left(3-2^{-(m-1)}\right)^{p-1}}\cdot
\left\|h\left(\frac{k}{2^m}\right)-L_h\left(\frac{k}{2^m}\right)\right\|_Y^p.
\end{equation}
Since $E_0(h)=\|h(1)-h(0)\|_Y^p$, the desired inequality~\eqref{eq:better inductive} holds as equality when $m=0$.  Fix $m\in \N$ and assume that~\eqref{eq:better inductive} holds true with $m$ replaced by $m-1$.

Convexity of $\|\cdot\|_Y^p$ implies that for every $m\in \N$,
\begin{eqnarray}\label{eq:monotone}
E_m(h)&=&\nonumber2^{m(p-1)}\sum_{k=0}^{2^m-1}\left\|h\left(\frac{k+1}{2^m}\right)-h\left(\frac{k}{2^m}\right)
\right\|_Y^p \\&=&\nonumber 2^{m(p-1)+1}\sum_{j=0}^{2^{m-1}-1}\frac{\left\|h\left(\frac{2j+1}{2^m}\right)-h\left(\frac{2j}{2^m}\right)
\right\|_Y^p
+\left\|h\left(\frac{2j+2}{2^m}\right)-h\left(\frac{2j+1}{2^m}\right)\right\|_Y^p}{2}\\
&\ge& 2^{m(p-1)+1}\sum_{j=0}^{2^{m-1}-1}
\left\|\frac{h\left(\frac{j+1}{2^{m-1}}\right)-h\left(\frac{j}{2^{m-1}}\right)}{2}\right\|_Y^p\nonumber\\&=&E_{m-1}(h).
\end{eqnarray}
Hence, if $k\in \{0,\ldots,2^{m}\}$ is even then by the inductive hypothesis
\begin{multline*}
E_m(h)\stackrel{\eqref{eq:monotone}}{\ge} E_{m-1}(h)\stackrel{\eqref{eq:better inductive}}{\ge} E_0(h)+\frac{2^p}{K^p\left(3-2^{-(m-2)}\right)^{p-1}}\left\|h\left(\frac{k/2}{2^{m-1}}
\right)-L_h\left(\frac{k/2}{2^{m-1}}\right)\right\|_Y^p\\
\ge E_0(h)+\frac{2^p}{K^p\left(3-2^{-(m-1)}\right)^{p-1}}
\left\|h\left(\frac{k}{2^{m}}\right)-L_h\left(\frac{k}{2^{m}}\right)\right\|_Y^p.
\end{multline*}
It therefore suffices to prove~\eqref{eq:better inductive} when $k$ is odd, say, $k=2j+1$. In this case, by reasoning analogously to~\eqref{eq:monotone}, we see that
\begin{eqnarray}\label{eq:use convexity}
&&\nonumber\!\!\!\!\!\!\!\!\!\!\!\!\frac{\!E_m(h)- E_{m-1}(h)}{2^{m(p-1)+1}}\\\nonumber\!\!\!\!\!&\ge&
\frac{\left\|h\left(\frac{2j+1}{2^m}\right)-h\left(\frac{2j}{2^m}\right)\right\|_Y^p
+\left\|h\left(\frac{2j+2}{2^m}\right)-h\left(\frac{2j+1}{2^m}\right)\right\|_Y^p}{2}
-\left\|\frac{h\left(\frac{2j+2}{2^m}\right)-h\left(\frac{2j}{2^m}\right)}{2}\right\|_Y^p\\
&\ge& \frac{1}{K^p}\left\|\frac{h\left(\frac{j}{2^{m-1}}\right)+h\left(\frac{j+1}{2^{m-1}}\right)}{2}-
h\left(\frac{k}{2^m}\right)\right\|_Y^p,
\end{eqnarray}
where in~\eqref{eq:use convexity} we used~\eqref{eq:4 point} with $$
x=\frac{h\left(\frac{2j+2}{2^m}\right)-h\left(\frac{2j}{2^m}\right)}{2}\quad
\mathrm{and}\quad
y= h\left(\frac{2j+1}{2^m}\right)-\frac{h\left(\frac{2j}{2^m}\right)+h\left(\frac{2j+2}{2^m}\right)}{2}.
$$

The inductive hypothesis implies that
\begin{multline}\label{eq:before holder}
\frac{K^p}{2^p}\left(E_{m-1}(h)-E_0(h)\right)\\\ge
\frac{\max\left\{\left\|h\left(\frac{j}{2^{m-1}}\right)-L_h\left(\frac{j}{2^{m-1}}\right)\right\|_Y^p
,\left\|h\left(\frac{j+1}{2^{m-1}}\right)-L_h\left(\frac{j+1}{2^{m-1}}\right)\right\|_Y^p\right\}}
{\left(3-2^{-(m-2)}\right)^{p-1}}.
\end{multline}
Since $L_h$ is affine, by convexity of $\|\cdot\|_Y^p$ we have
\begin{eqnarray}\label{eq:max on edges}
&&\!\!\!\!\!\!\!\!\!\!\!\!\!\!\nonumber
\left\|\frac{h\left(\frac{j}{2^{m-1}}\right)+h\left(\frac{j+1}{2^{m-1}}\right)}{2}
-L_h\left(\frac{k}{2^m}\right)\right\|_Y^p=
\left\|\frac{h\left(\frac{j}{2^{m-1}}\right)
-L_h\left(\frac{j}{2^{m-1}}\right)+h\left(\frac{j+1}{2^{m-1}}\right)-L_h\left(\frac{j+1}{2^{m-1}}\right)}
{2}\right\|_Y^p\\\nonumber
&\le& \frac{\left\|h\left(\frac{j}{2^{m-1}}\right)-L_h\left(\frac{j}{2^{m-1}}\right)\right\|_Y^p
+\left\|h\left(\frac{j+1}{2^{m-1}}\right)-L_h\left(\frac{j+1}{2^{m-1}}\right)\right\|_Y^p}{2}\\ &\le& \max\left\{\left\|h\left(\frac{j}{2^{m-1}}\right)-L_h\left(\frac{j}{2^{m-1}}\right)\right\|_Y^p
,\left\|h\left(\frac{j+1}{2^{m-1}}\right)-L_h\left(\frac{j+1}{2^{m-1}}\right)\right\|_Y^p\right\}.
\end{eqnarray}
Therefore, using~\eqref{eq:use convexity}, \eqref{eq:before holder} and \eqref{eq:max on edges}, we have
\begin{eqnarray}
&&\!\!\!\!\!\!\!\!\!\!\!\!\!\!\!\!\!\!\!\nonumber
\frac{K^p}{2^p}\left(E_m(h)-E_0(h)\right)\\&\ge&\frac{\left\|\frac{h\left(\frac{j}{2^{m-1}}\right)
+h\left(\frac{j+1}{2^{m-1}}\right)}{2}
-L_h\left(\frac{k}{2^m}\right)\right\|_Y^p}
{\left(3-2^{-(m-2)}\right)^{p-1}} +2^{(m-1)(p-1)}\left\|\frac{h\left(\frac{j}{2^{m-1}}\right)+h\left(\frac{j+1}{2^{m-1}}\right)}{2}
-h\left(\frac{k}{2^m}\right)\right\|_Y^p\nonumber\\\label{eq:holder}
&\ge& \frac{\left(\left\|\frac{h\left(\frac{j}{2^{m-1}}\right)+h\left(\frac{j+1}{2^{m-1}}\right)}{2}
-L_h\left(\frac{k}{2^m}\right)\right\|_Y
+\left\|\frac{h\left(\frac{j}{2^{m-1}}\right)+h\left(\frac{j+1}{2^{m-1}}\right)}{2}
-h\left(\frac{k}{2^m}\right)\right\|_Y\right)^p}{\left(3-2^{-(m-2)}+2^{-(m-1)}
\right)^{p-1}}\\\label{eq:done 1 dim}
&\ge& \frac{\left\|h\left(\frac{k}{2^m}\right)-L_h\left(\frac{k}{2^m}\right)\right\|_Y^p}{\left(3-2^{-(m-1)}\right)^{p-1}},
\end{eqnarray}
where in~\eqref{eq:holder} we used the inequality
$$
\forall \alpha,\beta,u,v\in (0,\infty),\quad \frac{u^p}{\alpha^{p-1}}+\frac{v^p}{\beta^{p-1}}\ge \frac{(u+v)^p}{(\alpha+\beta)^{p-1}},
$$
which is an immediate consequence of H\"older's inequality. Since inequality~\eqref{eq:done 1 dim} is the same as the desired inequality~\eqref{eq:better inductive}, the proof of Lemma~\ref{lem:use UC} is complete.
\end{proof}

\begin{proof}[Proof of Theorem~\ref{thm:main intro} when $n=1$] Our goal is to show that if $(Y,\|\cdot\|_Y)$ is a Banach space satisfying~\eqref{eq:4 point} then for every $\e\in (0,\frac12)$ we have
\begin{equation}\label{eq:R goal}
r^{\R\to Y}(\e) \ge \left(\frac{\e}{8}\right)^{(8K/\e)^p}.
\end{equation}
The fact that~\eqref{eq:R goal} is better than the desired estimate~\eqref{eq:our estimate r} is a simple elementary inequality (recall that $p\ge 2$, $K\ge 1$ and $0<\e<1/2$).

Assume for contradiction that~\eqref{eq:R goal} fails. Then there exists $\e\in (0,\frac12)$ and a $1$-Lipschitz function $h:[-1,1]\to Y$ such that for every $-1\le a<b\le 1$ with $b-a\ge (\e/8)^{(8K/\e)^p}$  there exists $t\in [a,b]$ satisfying $\|h(t)-L_h^{a,b}(t)\|_Y>\e(b-a)/2$. Choose $m\in \N$ such that $\e/8\le 2^{-m}<\e/4$ and take $k\in \{1,\ldots,m\}$ such that if we set $s=a+k2^{-m}(b-a)$ then $|s-t|\le (b-a)/2^{m+1}$. Because $f$ is $1$-Lipschitz, it follows immediately from the definition~\eqref{eq:def linear interpolation} of $L_h^{a,b}$  that it is also $1$-Lipschitz. Hence,
\begin{multline*}
\max_{k\in \{0,\ldots, 2^m\}} \frac{\left\|h\left(a+\frac{k}{2^m}(b-a)\right)-L_h^{a,b}\left(a+\frac{k}{2^m}(b-a)\right)\right\|_Y^p}
{(b-a)^p}\ge
\frac{\|h(s)-L_h^{a,b}(s)\|_Y^p}{(b-a)^p}
\\\ge \frac{\left(\|h(t)-L_h^{a,b}(t)\|_Y-\|h(t)-h(s)\|_Y-\|L_h^{a,b}(t)-L_h^{a,b}(s)\|_Y\right)^p}{(b-a)^p}
\ge \left(\frac{\e}{2}-\frac{1}{2^m}\right)^p\ge \frac{\e^p}{4^p}.
\end{multline*}
 Consequently, it follows from Lemma~\ref{lem:use UC} that
 \begin{equation}\label{eq:set up iteration}
-1\le a<b\le 1\ \wedge\ b-a\ge (\e/8)^{(8K/\e)^p}\implies E_m^{a,b}(h)\ge \frac{\|h(b)-h(a)\|_Y^p}{(b-a)^p}+\left(\frac{\e}{8K}\right)^p.
 \end{equation}

For $k\in \N\cup\{0\}$ and $j\in \{0,\ldots,2^{km}\}$ denote $a_j^k=-1+j/2^{km-1}$. If  $1/2^{km-1}\ge (\e/8)^{(8K/\e)^p}$ then it follows from~\eqref{eq:set up iteration} that for every $j\in \{0,\ldots,2^{km}-1\}$ we have
\begin{equation}\label{eq:on each interval}
E_m^{a_j^k,a_{j+1}^k}(h)\ge \frac{\left\|h\left(a_{j+1}^k\right)-h\left(a_{j}^k\right)\right\|_Y^p}{2^{(km-1)p}}+\left(\frac{\e}{8K}\right)^p.
\end{equation}
Hence,
\begin{multline}\label{eq:for summing}
E_{(k+1)m}^{-1,1}(h)\stackrel{\eqref{eq:def Eab}}
{=}2^{-km}\sum_{j=0}^{2^{km}-1}E_m^{a_j^k,a_{j+1}^k}(h)\\\stackrel{\eqref{eq:on each interval}}{\ge} 2^{-km} \sum_{j=0}^{2^{km}-1}\frac{\left\|h\left(a_{j+1}^k\right)-h\left(a_{j}^k\right)\right\|_Y^p}{2^{(km-1)p}}+
\left(\frac{\e}{8K}\right)^p\stackrel{\eqref{eq:def Eab}}
{=}E_{km}^{-1,1}(h)+\left(\frac{\e}{8K}\right)^p.
\end{multline}
Since $h$ is $1$-Lipschitz, the definition~\eqref{eq:def Eab} implies that $E_j^{a,b}(h)\le 1$ for all $-1\le a<b\le 1$ and $j\in \N$.  Denote $M=\left\lfloor \left(1+(8K/\e)^p\log_2(8/\e)\right)/m\right\rfloor$. Then~\eqref{eq:for summing} holds for every $k\in \N\cap [0,M]$. It follows that $E_{(M+1)m}^{-1,1}(h)\ge E_0^{-1,1}(h)+(M+1)(\e/(2K))^p\ge (M+1)(\e/(2K))^p$. Observe that the definition of $M$, combined with $2^{-m}\ge \e/8$, implies that $(M+1)(\e/(2K))^p>1$. Thus $E_{(M+1)m}^{-1,1}(h)>1$, a contradiction.
\end{proof}

\section{Proof of Theorem~\ref{thm:main intro} when $n\ge 2$}\label{sec:n>1}

Fix $n\in \N$ and let $(X,\|\cdot\|_X)$ be an $n$-dimensional normed space. Assume also that $(Y,\|\cdot\|_Y)$ is a Banach space and $f:X\to Y$. By John's theorem~\cite{Jo} there exists a norm $\|\cdot \|_2$ on $X$ which is Hilbertian and satisfies $\|x\|_2\le \|x\|_X\le \sqrt{n}\|x\|_2$ for all $x \in X$.  Let $\{e_1,\ldots,e_n\}$ be an orthonormal basis with respect to $\|x\|_2$. Via the obvious identifications, we may assume below that $X=\R^n$ and $\{e_1,\ldots,e_n\}$ is the standard coordinate basis.

For $y\in \R^n$ and $j\in \{1,\ldots, n\}$ define $f_j^y:\R\to Y$ by
$f_j^y(t)=f(y+te_j)$. Also, given $m\in \N$ and $j\in
\{1,\ldots,n\}$ set $F_j^m= \left\{z\in
\frac{1}{2^m}\{0,\ldots,2^m\}^n:\ z_j=0\right\}$. For $x\in \R^n$
and $\t\in (0,\infty)$ consider the following quantity
\begin{multline}\label{eq:def D}
\mathscr{D}_\t^m(f)(x)\eqdef\max_{\substack{j\in\{1,\ldots,n\}\\
y\in x+ \t F_j^m\\k\in
\{0,\ldots,2^m\}}}\frac{\left\|f\left(y+\frac{k\t}{2^{m}}e_j\right)-
f(y)-\frac{k}{2^{m}}\left(f\left(y+\t e_j\right)-f(y)\right)\right\|_X}{\t}\\
\stackrel{\eqref{eq:def linear
interpolation}}{=}\max_{\substack{j\in\{1,\ldots,n\}\\ y\in x+ \t
F_j^m\\k\in
\{0,\ldots,2^m\}}}\frac{1}{\t}\left\|f_j^y\left(\frac{k}{2^m}\t\right)-L_{f_j^y}^{0,\t}
\left(\frac{k}{2^m}\t\right) \right\|_Y.
\end{multline}

\begin{lemma}\label{lem:walsh}
Fix $x\in \R^n$, $m\in \N$ and $\e,\t\in (0,\infty)$ with $2^m\ge
2/\e\ge 10n^2$. Suppose that $f:\R^n\to Y$ satisfies
$\|f(y)-f(z)\|\le \|y-z\|_2$ for all $y,z\in x+[0,\t]^n$, i.e., $f$
is $1$-Lipschitz with respect to the Euclidean metric on the cube
$x+[0,\t]^n$. Suppose also that $\mathscr{D}_\t^m(f)(x)\le \e$. Then
there exists an affine mapping $A:\R^n\to Y$ such that
\begin{equation}\label{eq:approx linear part}
\sup_{z\in x+[0,\sqrt{\e}\t]^n}\left\|f(z)-A(z)\right\|_Y\le
8n^2\e\t.
\end{equation}
\end{lemma}

\begin{proof} By translation and rescaling we may assume without loss of generality that $x=0$ and $\t=1$. We will prove by induction on $n$ that there exist vectors $\{v_{S}\}_{S\subseteq \{1,\ldots,n\}}\subseteq Y$ with \begin{equation}\label{eq:coefficient bound}
v_\emptyset =f(0)\quad \mathrm{and}\quad \forall\  \emptyset\neq S \subseteq \{1,\ldots,n\},\quad \|v_S\|_Y\le 2^{|S|-1},
\end{equation}
such that for every $y\in \frac{1}{2^{m}}\{0,\ldots,2^m\}^n$ we have
\begin{equation}\label{eq:approx walsh}
\left\|f(y)-\sum_{S\subseteq \{1,\ldots,n\}}W_S(y)v_S\right\|_Y\le \e n,
\end{equation}
where the Walsh functions $\{W_S:\R^n\to \R\}_{S\subseteq \{1,\ldots,n\}}$ are defined as usual by $ W_S(y)\eqdef \prod_{i\in S} y_i$.

Assuming for the moment that this assertion has been proven, we proceed to deduce~\eqref{eq:approx linear part}.  Define $A:\R^n\to Y$ by $A(z)=v_\emptyset+\sum_{i=1}^n z_iv_{\{i\}}$. For $z\in [0,1]^n$ choose $y\in \frac{1}{2^m}\{0,\ldots,2^m-1\}^n$ with $|z_i-y_i|\le 1/2^{m+1}$ for all $i\in \{1,\ldots,n\}$. If we assume in addition that $z\in [0,\sqrt{\e}]^n$ then also $0\le y_i\le \frac{1}{2^{m+1}}+\sqrt{\e}$ for all $i\in \{1,\ldots,n\}$. Setting $g(y)=\sum_{S\subseteq \{1,\ldots,n\}}W_S(y)v_S$, we have
\begin{eqnarray}
&&\nonumber\!\!\!\!\!\!\!\!\!\!\!\!\!\|f(z)-A(z)\|_Y\\&\le&\nonumber \|f(z)-f(y)\|_Y+\left\|f(y)-g(y)\right\|_Y
+
\sum_{\substack{S\subseteq \{1,\ldots,n\}\\|S|\ge 2}}W_S(y)\left\|v_S\right\|_Y+\sum_{i=1}^n |z_i-y_i|\cdot\left\|v_{\{i\}}\right\|_Y\\ \label{eq:explain approx}
&\le& \frac{\sqrt{n}}{2^{m+1}}+n\e+\sum_{k=2}^n \binom{n}{k}\left(\sqrt{\e}+\frac{1}{2^{m+1}}\right)^k2^{k-1}+
\frac{n}{2^{m+1}}\\
&=&\nonumber \frac{\sqrt{n}+n}{2^{m+1}}+n\e+\frac12\left(\left(1+2\sqrt{\e}+\frac{1}{2^m}
\right)^n-1-2n\sqrt{\e}-\frac{n}{2^m}\right)\\
\label{eq:explain assumption+monotone}
&\le &3n\e+\frac{\left(1+\sqrt{5\e}\right)^n-1-n\sqrt{5\e}}{2}\\&\le&\label{eq:second order bound} 8n^2\e,
\end{eqnarray}
where in~\eqref{eq:explain approx} we used the fact that $f$ is $1$-Lipschitz and $\|y-z\|_2\le \sqrt{n}\|y-z\|_\infty\le \frac{\sqrt{n}}{2^{m+1}}$, the estimates~\eqref{eq:coefficient bound}, \eqref{eq:approx walsh}, and the above bounds on $\|y-z\|_\infty$ and $\|y\|_\infty$. In~\eqref{eq:explain assumption+monotone} we used our assumption $2^m\ge 2/\e\ge 10n^2$, which directly implies that $2\sqrt{\e}+2^{-m}\le \sqrt{5\e}\le 1/n$, together with the fact that the mapping $s\mapsto (1+s)^n-1-ns$ is increasing on $(0,\infty)$. In~\eqref{eq:second order bound} we used the elementary inequality $(1+s)^n-1-ns\le 2n^2s^2$, which is valid when $s\in (0,1/n)$.

It remains to prove~\eqref{eq:coefficient bound} and~\eqref{eq:approx walsh}, which will be done by induction on $n$. For $n=1$ set
$v_\emptyset= f(0)$ and  $v_{\{1\}}= f\left(1\right)-f(0)$.
 Since $f$ is $1$-Lipschitz we know that $\|v_{\{1\}}\|_Y\le 1$, proving~\eqref{eq:coefficient bound}. For the above choices of $v_\emptyset,v_{\{1\}}$, the estimate~\eqref{eq:approx walsh} is the same as the assumption $\mathscr{D}_1^m(f)(x)\le \e$ (recall that $\t=1$).

 If $n>1$  apply the inductive hypothesis to the functions $f_0,f_1:\R^{n-1}\to Y$ given by $f_0(y_1,\ldots,y_{n-1})=f(y_1,\ldots,y_{n-1},0)$ and $f_1(y_1,\ldots,y_{n-1})=f(y_1,\ldots,y_{n-1},1)$. One obtains $\{v^0_{S}\}_{S\subseteq \{1,\ldots,n-1\}}, \{v^1_{S}\}_{S\subseteq \{1,\ldots,n-1\}}\subseteq Y$ satisfying $v^0_\emptyset=f(0)$, $v^1_\emptyset=f(e_n)$, for all nonempty $S\subseteq \{1,\ldots,n-1\}$ we have $\|v^0_S\|_Y,\|v_S^1\|_Y\le 2^{|S|-1}$,  and if we define  $g_0,g_1:\R^{n-1}\to Y$ by
$$
g_i(y)\eqdef\sum_{S\subseteq \{1,\ldots,n-1\}}W_S(y)v^i_S,
$$
then
\begin{equation}\label{eq:inductive approx}
\max\left\{\|g_0(y)-f_0(y)\|_Y,\|g_1(y)-f_1(y)\|_Y\right\}\le \e(n-1)
 \end{equation}
 for all $y\in \frac{1}{2^m}\{0,\ldots,2^m\}^{n-1}$. For $S\subseteq \{1,\ldots,n\}$ define
\begin{equation}\label{eq:def new coef}
v_S\eqdef \left\{\begin{array}{ll}v_S^0& \mathrm{if\ }n\notin S,\\
v_{S\setminus\{n\}}^1-v_{S\setminus\{n\}}^0&\mathrm{if\ } n\in S.\end{array}\right.
\end{equation}
So, $v_\emptyset=v_\emptyset^0=f(0)$. If $S\neq\emptyset$ and $n\notin S$ then  have $\|v_S\|_Y=\|v_S^0\|_Y\le 2^{|S|-1}$. If $n\in S$ and $S\setminus\{n\}\neq\emptyset$ then
$
\|v_{S}\|_Y\le \|v_{S\setminus\{n\}}^0\|_Y+\|v^1_{S\setminus\{n\}}\|_Y\le
2^{|S|-2} =2^{|S|-1}.
$
Finally, since $f$ is $1$-Lipschitz we have
$
\|v_{\{n\}}\|_Y=\|v^1_\emptyset-v_\emptyset^0\|_Y=\|f( e_n)-f(0)\|_Y\le 1.
$
This completes the proof of~\eqref{eq:coefficient bound}. To prove~\eqref{eq:approx walsh} define for  $y\in \R^n$,
\begin{equation}\label{eq:decompose next g}
g(y)\eqdef\sum_{S\subseteq \{1,\ldots,n\}}W_S(y)v_S\stackrel{\eqref{eq:def new coef}}{=}\left(1-y_n\right)g_0(y_1,\ldots,y_{n-1})+y_ng_1(y_1,\ldots,y_{n-1}).
\end{equation}
The assumption $\mathscr{D}_1^m(f)(x)\le \e$ implies that for all $y\in \frac{1}{2^m}\{0,\ldots,2^m\}^{n-1}$ and all $k\in \{0,\ldots,2^m\}$ we have
\begin{equation}\label{eq:approx fiber}
\left\|f\left(y_1,\ldots,y_{n-1},\frac{k}{2^m}\right)-
\left(1-\frac{k}{2^m}\right)f_0(y)-\frac{k}{2^m}f_1(y)\right\|_Y\le \e.
\end{equation}
Hence,
\begin{multline}\label{eq:to say done}
\left\|f\left(y_1,\ldots,y_{n-1},\frac{k}{2^m}\right)-
g\left(y_1,\ldots,y_{n-1},\frac{k}{2^m}\right)\right\|_Y\\
\stackrel{\eqref{eq:decompose next g}\wedge\eqref{eq:approx fiber}}{\le} \e+\left(1-\frac{k}{2^m}\right)\left\|f_0(y)-g_0(y)\right\|_Y+\frac{k}{2^m}
\left\|f_1(y)-g_1(y)\right\|_Y\stackrel{\eqref{eq:inductive approx}}{\le} \e n.
\end{multline}
Since~\eqref{eq:to say done} holds for all $y\in \frac{1}{2^m}\{0,\ldots,2^m\}^{n-1}$ and all $k\in \{0,\ldots,2^m\}$, the proof of~\eqref{eq:approx walsh} is complete.
\end{proof}

\begin{proof}[Proof of Theorem~\ref{thm:main intro} when $n\ge 2$]
Our goal is to show that if $(Y,\|\cdot\|_Y)$ is a Banach space
satisfying~\eqref{eq:4 point} then for every $\e\in (0,\frac12)$ we
have
\begin{equation}\label{eq:R goaln}
r^{X\to Y}(\e) \ge R\eqdef \e^{K^pn^{20(n+p)}/\e^{2p+2n-2}}.
\end{equation}
Assume for contradiction that~\eqref{eq:R goaln} fails. Then there
exists $\e\in (0,\frac12)$ and a $\frac{1}{\sqrt{n}}$-Lipschitz
function $f:B_X\to Y$ such that for all $\rho\ge R$ and $y\in X$
such that $y+\rho B_X\subseteq B_X$, if $A:X\to Y$ is affine then
\begin{equation}\label{eq:contrapositive UAAP}
\sup_{z\in y+\rho B_X}\frac{\|f(z)-A(z)\|_Y}{\rho}>
\frac{\e}{\sqrt{n}}.
\end{equation}
 We claim that this implies the following statement.
\begin{equation}\label{eq:D big}
x\in \frac12B_X\ \wedge\
\t\in\left[\frac{32n^{5/2}}{\e}R,\frac{1}{2n}\right]\ \wedge\ 2^m\in
\left[ \frac{512n^5}{\e^2},\infty\right)\cap \N\implies
\mathscr{D}_\t^m(f)(x)>\frac{\e^2}{256 n^5}.
\end{equation}
Indeed, note that, because $\|\cdot\|\le \sqrt{n}\|\cdot \|_2\le
n\|\cdot\|_\infty$, the assumptions in~\eqref{eq:D big} imply
that $x+[0,\t]^n\subseteq  B_X$. Since $f$ is
$\frac{1}{\sqrt{n}}$-Lipschitz, it is $1$-Lipschitz with respect to
the Euclidean norm. If $\mathscr{D}_\t^m(f)(x)\le \e^2/(256 n^5)$
then it would follow from Lemma~\ref{lem:walsh} that there exists an
affine mapping $A:X\to Y$ such that
\begin{equation}\label{eq:sup on cube}
\sup_{z\in
x+\left[0,\e\t/(16n^{5/2})\right]}\frac{\left\|f(z)-A(z)\right\|_Y}{\e\t/(16n^{5/2})}\le
8n^2\sqrt{\frac{\e^2}{256 n^5}}=\frac{\e}{2\sqrt{n}}.
\end{equation}
Because $\|\cdot\|_X\ge \|\cdot \|_2$, we have $[-1,1]^n\supseteq
B_X$. Setting $\rho= \e\t/(32n^{5/2})\ge R$, it follows that
$x+\left[0,\e\t/(16n^{5/2})\right]^n\supseteq y+ \rho B_X$ for some
$y\in X$ with $y+\rho B_X\subseteq B_X$. Hence~\eqref{eq:sup on
cube} contradicts~\eqref{eq:contrapositive UAAP}, completing the
verification of~\eqref{eq:D big}.

It remains to argue that~\eqref{eq:D big} leads to a contradiction.
To this end, consider the following quantity, defined for every
$x\in \frac12 B_X$, $m,k\in \N\cup \{0\}$ and $\t\in (0,1/(2n)]$.
\begin{equation}\label{eq:two parameter H}
H^\t_{m,k}(f)(x)\eqdef\frac{1}{2^{m(n-1)}}\sum_{j=1}^n\sum_{\substack{y\in
\{0,\ldots,2^{m}-1\}^n\\y_j=0}} E_k^{0,\t}\left(f_j^{x+\t
2^{-m}y}\right).
\end{equation}
In~\eqref{eq:two parameter H}, recall the notation
$f_j^u(t)=f(u+te_j)$ and the definition~\eqref{eq:def Eab}. One checks directly from the definition~\eqref{eq:two parameter H} that the following recursive relation holds true. If $\alpha,\beta,\gamma\in \N\cup \{0\}$ and $\alpha\ge \beta $ then for every $\t\in (0,1/(2n)]$,
\begin{equation}\label{eq:recursive idenity}
H^\t_{\alpha,\beta+\gamma}(f)(0)=\frac{1}{2^{\beta n}}\sum_{x\in
\{0,\ldots,2^{\beta}-1\}^n}H^{\t/2^\beta}_{\alpha-\beta,\gamma}(f)\left(\frac{\t}{2^\beta}x\right).
\end{equation}

Observe that the fact that $f$ is
$\frac{1}{\sqrt{n}}$-Lipschitz and $\|e_j\|_X\le
\sqrt{n}\|e_j\|_2=\sqrt{n}$ implies that in each of the summands
in~\eqref{eq:two parameter H} the function $f_j^{x+\t
2^{-m}y}:[0,\t]\to Y$ is $1$-Lipschitz. Therefore we have the
point-wise bound $E_k^{0,\t}\left(f_j^{x+\t 2^{-m}y}\right)\le 1$
for each summand in~\eqref{eq:two parameter H}, implying that
\begin{equation}\label{eq:H a priori bound}
H^\t_{m,k}(f)(x)\le n.
\end{equation}

Set
\begin{equation}\label{eq:def m}
m\eqdef\left\lceil \log_2\left(\frac{512n^5}{\e^2}\right)\right\rceil,
\end{equation}
and
\begin{equation}\label{eq:def M high dim}
M\eqdef \left\lfloor
\frac{1}{m}\log_2\left(\frac{\e}{64n^{7/2}R}\right)\right\rfloor.
\end{equation}
Fix also an integer $k\in [0,M]$ and set
\begin{equation}\label{eq:def theta}
\t\eqdef\frac{1}{2^{km+1}n}.
\end{equation}
Then $\t\ge 32 n^{5/2}R/\e$ (recall~\eqref{eq:R goaln},
\eqref{eq:def m}, \eqref{eq:def M high dim}). It follows
from~\eqref{eq:D big} that $\mathscr{D}_{\t}^m(f)(x)\ge
\e^2/(2^8n^5)$. By the definition~\eqref{eq:def D}, this means that
there exists $j\in \{1,\ldots,n\}$ and $w\in x+\t F_j^m$ (recall
that $F_j^m= \left\{z\in \frac{1}{2^m}\{0,\ldots,2^m\}^n:\
z_j=0\right\}$), such that for some $s\in \{0,\ldots,2^m\}$ we have
\begin{equation}\label{eq:bad line}
\left\|f\left(w+\frac{s\t}{2^{m}}e_j\right)-f(w)-\frac{s}{2^m}
\left(f\left(w+\t e_j\right)-f(w)\right)\right\|_Y\ge
\frac{\e^2}{2^9n^62^{km}}.
\end{equation}

Denote $\ell=(M+1-k)m$ and consider the set
\begin{equation}\label{eq:def C}
C\eqdef\left\{y\in\left\{0,\ldots,2^{\ell}-1\right\}:\ y_j=0\
\wedge\ \left\|y-\frac{2^{\ell}}{\t}(w-x)\right\|_\infty\le
\frac{\e^22^{\ell}}{2^{10}n^{11/2}}\right\}.
\end{equation}
 Then
\begin{equation}\label{eq:C size}
|C|\ge \left\lfloor\frac{\e^22^\ell}{2^{10}n^{11/2}}\right\rfloor^{n-1}\\
\ge\left(\frac{\e^2}{2^{11}n^{11/2}}\right)^{n-1}2^{\ell(n-1)}.
\end{equation}
 Since the Lipschitz constant of $f$ with respect to
the $\ell_\infty$ norm is at most $\sqrt{n}$, it follows
from~\eqref{eq:bad line} that for every $y\in C$ we have
\begin{multline}\label{eq:nearby lines}
\left\|f\left(x+\frac{\t}{2^{\ell}}y+\frac{s\t}{2^{m}}e_j\right)-
f\left(x+\frac{\t}{2^{\ell}}y\right) -\frac{s}{2^m}
\left(f\left(x+\frac{\t}{2^{\ell}}y+\t e_j\right)-
f\left(x+\frac{\t}{2^{\ell}}y\right)\right)\right\|_Y\\\stackrel{\eqref{eq:def
C}}{\ge}
\frac{\e^2}{2^9n^62^{km}}-2\sqrt{n}\cdot\frac{\t}{2^\ell}\cdot
\frac{\e^22^{\ell}}{2^{10}n^{11/2}} \stackrel{\eqref{eq:def
theta}}{=} \frac{\e^2}{2^{10}n^62^{km}}.
\end{multline}
An equivalent way to write~\eqref{eq:nearby lines} is as follows.
\begin{equation*}\label{eq:equivalent nearby}
\left\|f_j^{x+\t 2^{-\ell}y}\left(\frac{s}{2^m}\right)-L_{f_j^{x+\t
2^{-\ell}y}}^{0,\t} \left(\frac{s}{2^m}\right)\right\|_Y\ge
\frac{\e^2}{2^{10}n^62^{km}}.
\end{equation*}
An application of Lemma~\ref{lem:use UC} now implies that for every
$y\in C$ we have
\begin{equation}\label{eq:lower in C}
\forall\ y\in C,\quad E_m^{0,\t}\left(f_j^{x+\t2^{-\ell}y}\right)\ge
\frac{\left\|f_j^{x+\t2^{-\ell}y}\left(\t\right)-f_j^{x+\t2^{-\ell}y}(0)\right\|_Y^p}{\t^p}+
\left(\frac{\e^2}{K(4n)^5}\right)^p,
\end{equation}
where $K$ is the constant in~\eqref{eq:4 point}. Also, by convexity (see~\eqref{eq:monotone}), for every $i\in \{1,\ldots, n\}$ we have
\begin{equation}\label{eq:lower convexity all y}
y\in \{0,\ldots,2^{\ell}\}\  \wedge\ y_i=0\implies E_m^{0,\t}\left(f_i^{x+\t2^{-\ell}y}\right)\ge
\frac{\left\|f_i^{x+\t2^{-\ell}y}\left(\t\right)-f_i^{x+\t2^{-\ell}y}(0)\right\|_Y^p}{\t^p}.
\end{equation}
Hence,
\begin{eqnarray}\label{eq:for recursion multi dim}
&&\nonumber\!\!\!\!\!\!\!\!\!\!\!\!\!\!\!\!\!\!\!\!\!H^\t_{\ell,m}(f)(x)\stackrel{\eqref{eq:two
parameter H}}{=}\frac{1}{2^{\ell(n-1)}}\sum_{y\in
C}E_m^{0,\t}\left(f_j^{x+\t
2^{-\ell}y}\right)+\frac{1}{2^{\ell(n-1)}} \sum_{\substack{y\in
\{0,\ldots,2^{\ell}-1\}^n\\y_j=0\\y\notin C}}
E_m^{0,\t}\left(f_j^{x+\t
2^{-\ell}y}\right)\\\nonumber&&+\frac{1}{2^{\ell(n-1)}}\sum_{\substack{i\in
\{1,\ldots,n\}\\i\neq j}}\sum_{\substack{y\in
\{0,\ldots,2^{\ell}-1\}^n\\y_i=0}} E_m^{0,\t}\left(f_i^{x+\t
2^{-\ell}y}\right)\\
&&\!\!\!\!\!\!\!\!\!\!\!\!\!\!\!\!\!\!\!\!\nonumber\stackrel{\eqref{eq:lower
in C}\wedge\eqref{eq:lower convexity all y}}{\ge}
\frac{1}{2^{\ell(n-1)}} \sum_{i=1}^n\sum_{\substack{y\in
\{0,\ldots,2^{\ell}-1\}^n\\y_i=0}}
\frac{\left\|f_i^{x+\t2^{-\ell}y}\left(\t\right)
-f_i^{x+\t2^{-\ell}y}(0)\right\|_Y^p}{\t^p}+
\frac{|C|}{2^{\ell(n-1)}}
\left(\frac{\e^2}{K(4n)^5}\right)^p\\
&&\!\!\!\!\!\!\!\!\!\!\!\!\!\!\!\stackrel{\eqref{eq:C size}}{\ge}
\frac{1}{2^{\ell(n-1)}} \sum_{i=1}^n\sum_{\substack{y\in
\{0,\ldots,2^{\ell}-1\}^n\\y_i=0}}
\frac{\left\|f\left(x+\frac{\t}{2^\ell}(y+2^\ell e_i)\right)
-f\left(x+\frac{\t}{2^\ell}y\right)\right\|_Y^p}{\t^p}
+\frac{\e^{2(n-1+p)}}{K^p(4n)^{6n+5p}}.
\end{eqnarray}

Now, using the recursive identity~\eqref{eq:recursive idenity}, we have
\begin{equation}\label{eq:use recursive}
H^{1/(2n)}_{(M+1)m,(k+1)m}(f)(0)=\frac{1}{2^{kmn}}\sum_{x\in
\{0,\ldots,2^{km}-1\}^n}H^{2^{-km}/(2n)}_{(M+1-k)m,m}(f)\left(\frac{2^{-km}}{2n}x\right).
\end{equation}
We relate~\eqref{eq:use recursive} to~\eqref{eq:for recursion multi dim} by noting the following identity, in which we recall that $\t$ is given in~\eqref{eq:def theta} and $\ell=(M+1-k)m$.
\begin{eqnarray}\label{partition identity}
&&\!\!\!\!\!\!\!\!\!\!\!\!\nonumber\frac{1}{2^{kmn+\ell(n-1)}}\sum_{x\in
\{0,\ldots,2^{km}-1\}^n} \sum_{i=1}^n\sum_{\substack{y\in
\{0,\ldots,2^{\ell}-1\}^n\\y_i=0}}
\frac{\left\|f\left(\frac{2^{-km}}{2n}x+\frac{\t}{2^\ell}(y+2^\ell
e_i)\right)-f\left(\frac{2^{-km}}{2n}x+\frac{\t}{2^\ell}y\right)\right\|_Y^p}{\t^p}\\\nonumber
&=& \frac{1}{2^{(M+1)m(n-1)+km}}\sum_{i=1}^n\sum_{z\in
\left\{0,\ldots,
2^{(M+1)m}-1\right\}^n}\frac{\left\|f\left(\frac{2^{-(M+1)m}}{2n}z+\frac{2^{-km}}{2n}e_i\right)
-f\left(\frac{2^{-(M+1)m}}{2n}z\right)\right\|_Y^p}{(2^{-km}/(2n))^p}\\\nonumber
&\stackrel{\eqref{eq:def Eab}}{=}&
\frac{1}{2^{(M+1)m(n-1)}}\sum_{i=1}^n \sum_{\substack{y\in
\left\{1,\ldots,2^{(M+1)m}-1\right\}\\y_i=0}}
E_{km}^{0,1/(2n)}\left(f_i^{\frac{2^{-(M+1)m}}{2n}y}\right)\\&\stackrel{\eqref{eq:two
parameter H}}{=}&H^{1/(2n)}_{(M+1)m,km}(f)(0).
\end{eqnarray}

By combining~\eqref{eq:for recursion multi dim}, \eqref{eq:use recursive} and~\eqref{partition identity} we conclude that
$$
\forall k\in \{0,\ldots, M\}, \quad
H^{1/(2n)}_{(M+1)m,(k+1)m}(f)(0)\ge
H^{1/(2n)}_{(M+1)m,km}(f)(0)+\frac{\e^{2(n-1+p)}}{K^p(4n)^{6n+5p}}.
$$
Hence,
\begin{equation}\label{eq:contradiction n}
n\stackrel{\eqref{eq:H a priori bound}}{\ge}
H^{1/(2n)}_{(M+1)m,(M+1)m}(f)(0)\ge
(M+1)\frac{\e^{2(n-1+p)}}{K^p(4n)^{6n+5p}}.
\end{equation}
Recalling the definitions~\eqref{eq:R goaln}, \eqref{eq:def m} and~\eqref{eq:def M high dim}, and that $K\ge 1$, $\e\in(0,\frac12)$ and $p,n\ge 2$, one checks that~\eqref{eq:contradiction n} is a contradiction.
\end{proof}

\section{An example}\label{sec:example}

We start with a simple one dimensional construction.
\begin{lemma}\label{lem:one dim counter}
Fix  $p\in [2,\infty)$ and $m\in \N$. There exists a $1$-Lipschitz function $f:[0,1]\to \ell_p^m$ with $f(0)=f(1)=0$ such that for every $0\le a<b\le 1$ with $b-a\ge 4/2^m$ and every affine mapping $A:\R\to \ell_p^m$ we have
\begin{equation*}
\sup_{x\in [a,b]}\frac{\|f_m(x)-A(x)\|_p}{(b-a)/2}> \frac{1}{8m^{1/p}}.
\end{equation*}
Consequently, if we set  $\e=\frac{1}{8m^{1/p}}$ then
$$
r^{\R\to \ell_p}(\e)\le \frac{4}{2^{1/(8\e)^p}}.
$$
\end{lemma}
\begin{proof}
Define inductively a sequence of functions $\{f_k:[0,1]\to \ell_p^m\}_{k=0}^m$ as follows. Let $\{e_1,\ldots,e_m\}$ be the standard basis of $\ell_p^m$. Set $f_0\equiv 0$. Assume that $k\in \N$ and we have defined $f_{k-1}$ to be affine on each of the dyadic intervals $\{[j/2^{k-1},(j+1)/2^{k-1}]\}_{j=0}^{2^{k-1}-1}$. For every $j\in \{0,\ldots,2^{k-1}\}$ define $f_k(j/2^{k-1})=f_{k-1}(j/2^{k-1})$ and
\begin{equation}\label{eq:def f_k}
f_k\left(\frac{2j+1}{2^k}\right)=f_{k-1}\left(\frac{2j+1}{2^k}\right)+\frac{1}{m^{1/p}2^k}e_k.
\end{equation}
Let $f_k$ be the piecewise affine extension of the above values of $f_k$ on $\{j/2^k\}_{k=0}^{2^k-1}$. A straightforward induction shows that
$$
\left\|f_k\left(\frac{j+1}{2^k}\right)-f_k\left(\frac{j}{2^k}\right)\right\|_p
=\frac{1}{2^k}\left(\frac{k}{m}\right)^{1/p}.
$$
Thus $f_m$ is $1$-Lipschitz.

Assume for contradiction that $0\le a<b\le 1$ satisfy $b-a\ge 4/2^m$, and there exists an affine mapping $A:\R\to \ell_p$ such that
\begin{equation}\label{eq:approx}
\sup_{x\in [a,b]}\frac{\|f_m(x)-A(x)\|_p}{(b-a)/2}\le \frac{1}{8m^{1/p}}.
\end{equation}
There exists $k\in \{1,\ldots,m\}$ such that $4/2^{k}\le b-a<8/2^k$. Because $b-a\ge 4/2^k$ there is $j\in \{0,\ldots,2^{k-1}-1\}$ such that $[j/2^{k-1},(j+1)/2^{k-1}]\subseteq [a,b]$. Now, since $A$ is affine and $f_{k-1}$ is affine on $[j/2^{k-1},(j+1)/2^{k-1}]$,
\begin{multline*}
\frac{b-a}{8m^{1/p}}\stackrel{\eqref{eq:approx}}{\ge} \left\|f_m\left(\frac{j/2^{k-1}+(j+1)/2^{k-1}}{2}\right)-\frac{f_m\left(j/2^{k-1}\right)+
f_m\left((j+1)/2^{k-1}\right)}{2}\right\|_p\\=\left\|f_k\left(\frac{2j+1}{2^k}\right)
-f_{k-1}\left(\frac{2j+1}{2^k}\right)\right\|_p\stackrel{\eqref{eq:def f_k}}{=}\frac{1}{m^{1/p}2^k},
\end{multline*}
in contradiction to the fact that $b-a<8/2^k$.
\end{proof}

\begin{lemma}\label{lem:localized 1 dim}
Fix  $p\in [2,\infty)$ and $m,n\in \N$. There exists a $1$-Lipschitz function $g:\R\to \ell_p^{m+1}$ such that for every $y\in \R$ with $|y|\le 1/\sqrt{n}$, every $r\ge 32/\left(\sqrt{n}2^{m}\right)$, and every affine mapping $A:\R\to \ell_p^{m+1}$,
\begin{equation*}\label{eq:desired localized}
\sup_{x\in [y-r,y+r]}\frac{\left\|g(x)-A(x)\right\|_p}{r}>\frac{1}{16m^{1/p}}.
\end{equation*}
\end{lemma}

\begin{proof} Let $\{e_1,\ldots,e_{m+1}\}$ denote the standard basis of $\ell_p^{m+1}$. Define $g:\R\to \ell_p^{m+1}$ by
$$
g(x)\eqdef \left\{\begin{array}{ll}\frac{4}{\sqrt{n}}f\left(\frac{\sqrt{n}}{4}x+\frac12\right)&\mathrm{if}\ |x|\le \frac{2}{\sqrt{n}},\\
\left(|x|-\frac{2}{\sqrt{n}}\right)e_{m+1}&\mathrm{otherwise}.\end{array}\right.
$$
where $f=f_m:[0,1]\to \ell_p^n=\mathrm{span}(\{e_1,\ldots,e_m\})$ is the function from Lemma~\ref{lem:one dim counter}. Because $f$ is $1$-Lipschitz and $f(0)=f(1)=0$, one checks that $g$ is $1$-Lipschitz.

Fix an affine mapping $A:\R\to \ell_p$ and take $y\in \R$ satisfying $|y|\le 1/\sqrt{n}$. Suppose that $r\ge 32/\left(\sqrt{n}2^{m}\right)$. If in addition $r\le 8/\sqrt{n}$ then write $[y-r,y+r]\cap [-2/\sqrt{n},2/\sqrt{n}]=[a,b]$, where $b-a\ge r/2\ge 16/\left(\sqrt{n}2^{m}\right)$. By Lemma~\ref{lem:one dim counter},
\begin{multline*}
\sup_{x\in [y-r,y+r]}\frac{\left\|g(x)-A(x)\right\|_p}{r}\ge \sup_{x\in [a,b]}\frac{\left\|\frac{4}{\sqrt{n}}f\left(\frac{\sqrt{n}}{4}x+\frac12\right)-A(x)\right\|_p}{2(b-a)}\\=
\frac12\sup_{z\in\left[\frac{\sqrt{n}}{4}a+\frac12,\frac{\sqrt{n}}{4}b+\frac12\right]}
\frac{\left\|f(z)-\frac{\sqrt{n}}{4}A\left(\frac{4}{\sqrt{n}}z-\frac{2}{\sqrt{n}}\right)\right\|_p}
{\left(\frac{\sqrt{n}}{4}b-\frac{\sqrt{n}}{4}a\right)/2}>\frac{1}{16m^{1/p}}.
\end{multline*}

It remains to deal with the case $r> 8/\sqrt{n}$. In this case $y-r,y+r\notin [-2/\sqrt{n},\sqrt{n}]$, so
$$
\left\langle g(y\pm r),e_{m+1}\right\rangle=|y\pm r|-\frac{2}{\sqrt{n}}\ge r-|y|-\frac{2}{\sqrt{n}}\ge r-\frac{3}{\sqrt{n}}>\frac{5}{\sqrt{n}}>0.
 $$
 Assume for contradiction that $\|g(x)-A(x)\|_p\le r/(16 m^{1/p})$ for all $x\in [y-r,y+r]$. Then, since $A$ is affine,
$$
 \left\langle A(y),e_{m+1}\right\rangle =\frac{ \left\langle A(y+r),e_{m+1}\right\rangle+ \left\langle A(y-r),e_{m+1}\right\rangle}{2}\ge r-\frac{3}{\sqrt{n}}- \frac{r}{16m^{1/p}}.
$$
Hence,
$$
\left\langle g(y),e_{m+1}\right\rangle \ge  \left\langle A(y),e_{m+1}\right\rangle-\|g(y)-A(y)\|_p\ge \left(1-\frac{1}{8m^{1/p}}\right)r-\frac{3}{\sqrt{n}}>\frac{r}{2}-\frac{3}{\sqrt{n}}>\frac{1}{\sqrt{n}},
$$
contradicting the fact that, since $|y|\le 2/\sqrt{n}$, we have $\left\langle g(y),e_{m+1}\right\rangle=0$.
\end{proof}

We now use the function $g$ of Lemma~\ref{lem:localized 1 dim} as a building block of a function $F:\ell_2^n\to \ell_2^n(\ell_p^{m+1})$ whose affine approximability properties deteriorate with the dimension $n$. This step is similar to an argument in the proof of Theorem~2.7 in~\cite{BaJoLi}.

\begin{lemma}\label{lem:n-dim example}
For every  $p\in [2,\infty)$ and every $m,n\in \N$ there exists a $1$-Lipschitz function $F:\ell_2^n\to \ell_2^n(\ell_p^{m+1})$ such that for every $$r\ge \frac{32}{\sqrt{n}2^{m}}$$ and every affine mapping $A:\ell_2^n\to \ell_2^n(\ell_p^{m+1})$,
$$
\sup_{x\in y+r B_{\ell_2^n}}\frac{\left\|F(x)-A(x)\right\|_{\ell_2^n(\ell_p^{m+1})}}{r}>\frac{1}{16m^{1/p}}.
$$
Consequently, if we set $\e= \frac{1}{16m^{1/p}}$ then for $X=\ell_2^n$ and $Y=\ell_2^n(\ell_p^{m+1})$,
$$
r^{X\to Y}(\e)\le \frac{32}{\sqrt{n} 2^{1/(16\e)^p}}.
$$
\end{lemma}

\begin{proof}
Let $g:\R\to \ell_p^{m+1}$ be the function from Lemma~\ref{lem:localized 1 dim}. Since $g$ is $1$-Lipschitz, if we define $F(x_1,\ldots,x_n)=(g(x_1),\ldots,g(x_n))$ then $F:\ell_2^n\to \ell_2^n(\ell_p^{m+1})$ is $1$-Lipschitz. Fixing $r\ge 32/\left(\sqrt{n}2^{m}\right)$, suppose that $A:\ell_2^n\to \ell_2^n(\ell_p^{m+1})$ is affine and $y+rB_{\ell_2^n}\subseteq B_{\ell_2^n}$. Since $\|y\|_{2}<1$, there exists $i\in \{1,\ldots,n\}$ such that $|y_i|<1/\sqrt{n}$. Writing $A(x)=(A_1(x),\ldots,A_n(x))$, define $A_i':\R\to \ell_p^{m+1}$  by $A_i'(t)=A_i\left(\sum_{j\in \{1,\ldots,n\}\setminus \{i\}}y_je_j+t e_i\right)$. By Lemma~\ref{lem:localized 1 dim} we know that
$$
\sup_{x\in y+r B_{\ell_2^n}}\frac{\left\|F(x)-A(x)\right\|_{\ell_2^n(\ell_p^{m+1})}}{r}\ge \sup_{t\in [y_i-r,y_i+r]}\frac{\left\|g(t)-A_i'(t)\right\|_{p}}{r}>\frac{1}{16m^{1/p}},
$$
where we used that fact that $y+[-r,r]e_i\subseteq y+r B_{\ell_2^n}$.
\end{proof}


\bibliographystyle{abbrv}
\bibliography{uaap}
\end{document}